\documentclass[11pt]{amsart}
\font\bbbld=msbm10 scaled\magstephalf

\def\c{\mathop{cap}}

\newcommand{\bC}{\hbox{\bbbld C}}

\newcommand{\bR}{\hbox{\bbbld R}}

\newcommand{\cE}{\mathcal{E}}

\newtheorem{theorem}{Theorem}[section]
\newtheorem{lemma}[theorem]{Lemma}
\newtheorem{proposition}[theorem]{Proposition}
\newtheorem{corollary}[theorem]{Corollary}

 \theoremstyle{definition}
\newtheorem{definition}[theorem]{Definition}

\theoremstyle{remark}
\newtheorem{remark}[theorem]{Remark}

\numberwithin{equation}{section}

\begin{document}
\setlength{\baselineskip}{1.2\baselineskip}

\title{Subsolution theorem  for the complex Hessian  equation}

\author{Ngoc Cuong Nguyen}
\address{Institute of Mathematics, Jagiellonian University,   \L ojasiewicza 6, 30-348 Krak\'ow, Poland. }
\email{Nguyen.Ngoc.Cuong@im.uj.edu.pl}

\begin{abstract}
We prove the subsolution theorem for the complex Hessian equation
in a smoothly bounded strongly $m$-pseudoconvex domain in $\bC^n$.

\end{abstract}

\maketitle

\bigskip

\section*{Introduction}
\label{Intro}
\setcounter{equation}{0}
Let $\Omega$ be a bounded domain in $\bC^n$ with the canonical
K\" ahler form $\beta= dd^c\|z\|^2$, where $d= \partial + \bar{\partial}$,
$d^c=i(\bar{\partial}- \partial)$. For $1\leq m\leq n$, we denote $\bC_{(1,1)}$
the space of $(1,1)$-forms with constant coefficients. One defines the positive cone
\begin{equation}
\label{in-1}
\Gamma_m=
	\{ \eta\in \bC_{(1,1)}: \eta\wedge \beta^{n-1}
	\geq 0, ... , \eta^m \wedge\beta^{n-m}\geq 0\}.
\end{equation}
A $C^2$ smooth function $u$ is called  $m$-subharmonic  in $\Omega$
if at every point $z\in\Omega$ the $(1,1)$-form associated to its complex
Hessian belongs to $\Gamma_m$, i.e
\begin{equation}
\label{in-2}
 \sum_{j,k=1}^n
 	\frac{\partial^2u(z)}{\partial z_j\partial\bar{z_k}} i dz_j\wedge d\bar{z_k}
	\in \Gamma_m.
\end{equation}
It was observed by B\l ocki (see \cite{Bl}) that one may relax the smoothness
condition in the definition \eqref{in-2} and consider this inequality in the sense
of distributions to obtain a class, denoted by $SH_m(\Omega)$ (see preliminaries).
When functions $u_1, ..., u_k$, $1\leq k\leq m$, are in $SH_m(\Omega)$ and are locally
bounded, one still may define
$dd^cu_1\wedge dd^cu_2\wedge ...\wedge dd^cu_k\wedge \beta^{n-m}$
as a closed positive current of bidegree $(n-m+k,n-m+k)$. In particular
 $(dd^cu)^m\wedge \beta^{n-m}$ is a positive measure for $u$ bounded $m$-subharmonic.
 Thus, it is possible to study bounded solutions of the Dirichlet problem with positive
 Borel measures $\mu$ in $\Omega$ and continuous boundary data
 $\varphi\in C(\partial \Omega)$:
\begin{equation}
\label{heq}
\begin{cases}
	u \in  SH_m(\Omega) \cap L^{\infty}(\Omega), \\
	(dd^cu)^m \wedge \beta^{n-m} = d\mu, \\
	u(z) = \varphi(z)  \;\; \text{ on } \;\; \partial \Omega.
\end{cases}
\end{equation}
The Dirichlet problem for the complex Hessian equation \eqref{heq} in smooth cases was
first considered by S.Y. Li (see \cite{Li}). His main result says that if $\Omega$ is smoothly
bounded and strongly $m$-pseudoconvex  (see Definition~\ref{pr-de-4}) then, for a smooth
boundary data and for a smooth positive  measure, i.e $d\mu=f\beta^n$ and $f>0$ smooth,
there exists a unique smooth solution of the Dirichlet problem for the  Hessian equation.

The weak solutions of the equation \eqref{heq}, when the measure $d\mu$
is  possibly degenerate, were first considered by B\l ocki \cite{Bl}, more precisely,
he proved that there exists  a unique continuous solution of the homogeneous
Dirichlet problem  in the unit ball in $\bC^n$.

Very recently, in \cite{DK} Dinew and Ko\l odziej investigated weak solutions of the complex
Hessian equations \eqref{heq} with the right hand side more general, namely $d\mu = f\beta^n$
where $f\in L^p$, for $p>n/m$. One of their results extended Li's theorem, they proved that
the Dirichlet problem still has a unique continuous solution provided continuous boundary data
and $d\mu$ in $ L^p$ as above.
Their method exploited the new counterpart of pluripotential theory
for $m$-subharmonic functions, after showing a crucial inequality between the usual volume and
$m$-capacity which is a version of  the relative capacity for $m$-subharmonic functions

In the case m=n, the subsolution theorem due to Ko\l odziej  \cite{K1}
(see \cite{K2} for a simpler proof)
says that the Dirichlet problem \eqref{heq} in a strongly pseudoconvex domain
is solvable if there is a subsolution. Thus, one may ask the same question when $m<n$.
In this note we show that the subsolution theorem, Theorem~\ref{su-th-1},
 for the complex Hessian equation is still true by combining the new results of
 Dinew and Ko\l odziej for weak solutions of the complex Hessian equations
 and the method used to prove the subsolution theorem in the pluripotential case.

\bigskip

\section*{Acknowledgements} I am indebted to my advisor, professor S\l awomir Ko\l odziej,
for suggesting the problem and for many stimulating discussions.
I also would like to thank the referee whose suggestions and remarks helped to improve the exposition of the
paper. This work is supported by the International Ph.D Program
{\em " Geometry and Topology in Physical Models "}.

\bigskip

\section{Preliminaries}
\label{pr}
\setcounter{equation}{0}

\subsection{$m$-subharmonic functions}
\label{pr1}
We recall basic notions and results which are adapted from pluripotential theory.
The main sources are \cite{BT1,BT2}, \cite{Ce1,Ce2},  \cite{D1, D2}, \cite{K2}
for plurisubharmonic functions and \cite{Bl}, \cite{DK} for $m$-subharmonic functions.
Since a major part of pluripotential theory can be easily adapted to $m$-
subharmonic case, when the proof is only a copy of the original one
with obvious changes of notations, for the
proofs we refer the reader to the above references.
Let $\bC_{(k,k)}$ be the space of $(k,k)$-forms with constant coefficients, and
$$
	\Gamma_m=
		\{ \eta\in \bC_{(1,1)}: \eta\wedge \beta^{n-1}\geq 0
					, ... , \eta^m \wedge\beta^{n-m}\geq 0\} .
$$
We denote by $ \Gamma_m^*$  its dual cone
\begin{equation}
\label{pr1-1}
\Gamma_m^*
	= \{ \gamma\in \bC_{(n-1,n-1)}: \gamma\wedge \eta
		\geq 0 \;\; \text{ for every } \;\; \eta\in \Gamma_m\}.
\end{equation}
By Proposition 2.1 in \cite{Bl} we know that
$\{\eta_1\wedge...\wedge\eta_{m-1}\wedge\beta^{n-m}; \;\;
\eta_1,...,\eta_{m-1}\in \Gamma_m\}\subset\Gamma_m^*$, moreover if we consider
$
\Gamma_m^{**}
	=\{\eta\in \bC_{(1,1)}: \eta\wedge\gamma\geq 0
		\;\; \text{for every} \;\; \gamma\in\Gamma_m^*\}
$
then we have
$$
	\Gamma_m=\Gamma_m^{**}
$$
as $ \{\eta_1\wedge...\wedge\eta_{m-1}\wedge\beta^{n-m};
	\;\; \eta_1,...,\eta_{m-1}\in \Gamma_m\}^*\subset \Gamma_m$.
Therefore
\begin{equation}
\label{pr1-2}
\Gamma_m^*
	=\{\eta_1\wedge...\wedge\eta_{m-1}\wedge\beta^{n-m};
		\;\; \eta_1,...,\eta_{m-1}\in \Gamma_m\}.
\end{equation}
Since $\Gamma_n\subset\Gamma_{n-1}\subset ... \subset\Gamma_1$, we thus obtain
$$
	\Gamma_n^*\supset \Gamma_{n-1}^*\supset ...\supset \Gamma_1^*
		=\{ t\beta^{n-1}; t\geq 0\}.
$$
In particular, when  $\eta\in\Gamma_m^*$, and  it has a representation
$$
	\sum  a^{j\bar k} i^{(n-1)^2}\hat{dz_j} \wedge\hat{d\bar{z_k}}
$$
(this notation means that in the $(n-1,n-1)$-form only $dz_j$ and $d\bar{z_k}$ disappear in the
complete form  $dz\wedge d\bar{z}$ at positions $j$-th and $k$-th) then
 the Hermitian matrix $(a^{j\bar k})$ is
nonnegative  definite. In the  language of differential forms, a $C^2$ smooth function $u$ is
$m$-subharmonic  ($m$-sh for short)  if
 $$
 dd^cu\wedge\beta^{n-1}
 	\geq 0, ...,(dd^cu)^m\wedge\beta^{n-m}\geq 0
		\;\; \text{ at every point in }\;\; \Omega.
$$
\begin{definition}
\label{pr-de-1}
Let $u$ be a subharmonic function on an open subset $\Omega\subset \bC^n$.
Then $u$ is called $m$-subharmonic if for any collection of
 $\eta_1,...,\eta_{m-1}$ in $\Gamma_m$, the inequality
$$
	dd^cu\wedge \eta_1\wedge ... \wedge\eta_{m-1}\wedge\beta^{n-m}
		\geq 0
$$
holds in the  sense of currents.
Let $SH_m(\Omega)$ denote the set of all $m$-sh functions in $\Omega$.
\end{definition}

\begin{remark}
\label{pr-re-2}
{\bf (a)}  The condition \eqref{pr1-1} is equivalent to
$dd^cu\wedge \eta\geq 0$ for every $\eta\in\Gamma_m^*$ by \eqref{pr1-2}.
Hence, a subharmonic function $u$ is $m$-subharmonic if
\begin{equation}
\label{pr-re}
\int_\Omega u~dd^c\phi\wedge\eta
	= \int_\Omega u
		\sum_{j,k=1}^na^{j\bar k} \frac{\partial^2\phi}{\partial z_j\partial\bar{z_k}} \beta^n
	\geq 0
\end{equation}
for every non-negative test function $0\leq\phi$ in $\Omega$
and for every nonnegative definite Hermitian matrix
$(a^{j\bar{k}})$ of constant coefficients such that
$\eta=
	\sum_{j,k=1}^na^{j\bar k} i^{(n-1)^2} dz_1\wedge ...
	\wedge\hat{dz_j}\wedge ...\wedge dz_n\wedge d\bar{z_1}\wedge ...
	\wedge\hat{d\bar{z_k}}\wedge ...\wedge d\bar{z_n}
$
belongs to $\Gamma_{m}^*$.
This means that $u$ is subharmonic with respect to
a family of elliptic operators with constant coefficients.

{\bf (b)} A $C^2$ function $v$ is $m$-subharmonic
iff $dd^cv(z)$ belongs to $\Gamma_m$ at every $z\in \Omega$. Hence
$$
	dd^cu\wedge dd^cv_1\wedge ... \wedge dd^cv_{m-1}\wedge\beta^{n-m}
		\geq 0
$$
holds in $\Omega$ in the weak sense of currents,
for every collection $v_1, ... ,v_{m-1}\in SH_m\cap C^2(\Omega)$ and any
$u\in SH_m(\Omega)$.
\end{remark}

\begin{proposition}
\label{pr-pr-3}
Let $\Omega\subset \bC^n$ be a bounded open subset. Then
\begin{enumerate}
\item
\label{pr-pr3-1}
$ PSH(\Omega)
	= SH_n (\Omega)\subset SH_{n-1} (\Omega)\subset\cdots \subset SH_1 (\Omega)
	= SH(\Omega)$.
\item
\label{pr-pr3-2}
$SH_m (\Omega)$ is a convex cone.	
\item
\label{pr-pr3-3}
The limit of a decreasing sequence in $SH_m(\Omega)$ belongs to $SH_m(\Omega)$.
Moreover, the standard regularization $u\ast\rho_{\varepsilon}$
of a $m$-sh function is again a $m$-sh fucntion.
There $\rho_\varepsilon(z)=\frac{1}{\varepsilon^{2n}}\rho(\frac{z}{\varepsilon})$,
$\rho(z)=\rho(\|z\|^2)$ is a smoothing kernel,
with $\rho: \bR_+\rightarrow \bR_+$ defined by
$$
\rho(t) =
	\begin{cases}
		\frac{C}{(1-t)^2} \exp(\frac{1}{t-1}) &\ {\rm if}\ 0\leq t\leq 1, \\
		0 &\ {\rm if}\  t>1,
	\end{cases}
$$
for a constant C such that
$$
	\int_{\bC^n} \rho(\|z\|^2)\beta^n=1.
$$
\item
\label{pr-pr3-4}
If $u\in SH_m (\Omega)$\ and $\gamma: \bR \rightarrow \bR$\ is a convex,
nondecreasing function then $\gamma\circ u\in SH_m(\Omega)$.	
\item
\label{pr-pr3-5}
If $u, v\in SH_m(\Omega)$ then $\max\{u,v\}\in SH_m(\Omega)$.	
\item
\label{pr-pr3-6}
Let $\{u_\alpha\}\subset SH_m(\Omega)$ be a locally uniformly bounded from above and
$u= \sup u_\alpha$. Then the upper semi-continuous regularization
$u^*$ is $m$-sh and is equal to $u$ almost everywhere.
\end{enumerate}
\end{proposition}

\begin{proof}
\eqref{pr-pr3-1} and \eqref{pr-pr3-2} and the first part of \eqref{pr-pr3-3}
are obvious from the definition of $m$-sh functions.
From the formula \eqref{pr-re} we have, for $\eta \in \Gamma_m^*$,
$$
	\int (u\ast\rho_{\varepsilon})~dd^c\phi\wedge\eta
		= \int u~dd^c(\phi\ast\rho_{\varepsilon})\wedge \eta\geq 0,
$$
since $\phi\ast\rho_{\varepsilon}$ is again a nonnegative test function.
Thus \eqref{pr-pr3-3} is proved.
For \eqref{pr-pr3-4},  the smooth function $\gamma\ast\rho_\varepsilon$
(the standard regularization on $\bR$)
is convex and increasing, therefore $(\gamma\ast\rho_\varepsilon)\circ u\in SH_m(\Omega)$.
Since
$(\gamma\ast\rho_\varepsilon)\circ u$ decreases to $\gamma\circ u$
as $\varepsilon\rightarrow 0$, applying
the first part of \eqref{pr-pr3-3} we have $\gamma\circ u\in SH_m(\Omega)$.
In order to prove \eqref{pr-pr3-5},
note that by using \eqref{pr-pr3-3} it is enough to show that
$w=\max \{u_\varepsilon, v_\varepsilon\}$ is $m$-sh,
where $u_\varepsilon:= u*\rho_\varepsilon, v_\varepsilon:= v*\rho_\varepsilon$.
Since $w$ is semi-convex,  i.e there is a constant $C=C_\varepsilon >0$ big enough such that
$w+C\|z\|^2=\max \{u_\varepsilon+C\|z\|^2, v_\varepsilon+C\|z\|^2\}$ is a convex function
in $\bR^{2n}$, hence
it has second derivative almost everywhere and $dd^c w(x) \in \Gamma_m$ for almost
everywhere $x $ in $\Omega$.
Let $w_\varepsilon$ is a regularization of $w$,
by the formula of the convolution
$w_\varepsilon(x) = \int_\Omega w(x-\varepsilon y)\rho(y)\beta^n(y)$ we have
$$
	dd^c w_\varepsilon(x) = \int_\Omega dd^c w(x-\varepsilon y)\rho(y)\beta^n(y).
$$
Thus, for $\eta \in \Gamma_m^*$
$$
	dd^c w_\varepsilon(x)\wedge\eta
		= \int_\Omega\left[ dd^c w(x-\varepsilon y)\wedge\eta\right] \rho(y)\beta^n(y)
		\geq 0.
$$
\eqref{pr-pr3-6} is a consequence of \eqref{pr-pr3-5} and Choquet's Lemma.
\end{proof}

\subsection{The complex Hessian operator}
\label{pr2}
For $1\leq k\leq m$, $u_1, ..., u_k\in SH_m\cap  L^\infty_{loc}(\Omega)$ the operator
$dd^cu_k\wedge dd^cu_{k-1}\wedge ... \wedge dd^cu_1\wedge\beta^{n-m}$
is defined inductively by (see \cite{Bl}, \cite{DK})
$$
	dd^cu_k\wedge dd^cu_{k-1}\wedge ... \wedge dd^cu_1\wedge\beta^{n-m}
		:= dd^c(u_kdd^cu_{k-1}\wedge ... \wedge dd^cu_1\wedge\beta^{n-m})
		\leqno(H_k)
$$
which is a closed positive current of bidegree $(n-m+k,n-m+k)$.
This operator is also continuous under
decreasing sequences and symmetric (see Remark~\ref{pr-reth9}).
In the case $k=m$, $dd^cu_1\wedge dd^cu_2\wedge ... \wedge dd^cu_m\wedge\beta^{n-m}$
is a nonnegative Borel measure, in particular, when $u=u_1=...=u_m$ currents (measures)
 $(dd^cu)^m\wedge\beta^{n-m}$ are well-defined for $u\in L_{loc}^\infty(\Omega)$.
 The above definitions essentially follow from the analogous definitions of
  Bedford and Taylor (\cite{BT1}, \cite{BT2}) for plurisubharmonic functions.
\begin{proposition}[Chern-Levine-Nirenberg inequalities]
\label{cln}
Let $K\subset\subset U \subset\subset \Omega$, where $K$ is compact, $U$ is open.
Let $u_1,...,u_k\in SH_m\cap  L^\infty(\Omega)$, $1\leq k\leq m$
and $v\in SH_m(\Omega)$ then there exists a
constant $C=C_{K,U,\Omega}\geq 0$ such that
\begin{enumerate}
 \item[(i)]
 	$\|dd^cu_1\wedge ... \wedge dd^cu_k\wedge\beta^{n-m}\|_{K}
 		\leq C~\|u_1\|_{ L^\infty(U)} ... \|u_k\|_{ L^\infty(U)},$
 \item[(ii)]
	 $\|dd^cu_1\wedge ... \wedge dd^cu_k\wedge\beta^{n-m}\|_{K}
 		\leq C~\|u_1\|_{ L^1(\Omega)}.\|u_2\|_{ L^\infty(\Omega)} ... \|u_k\|_{ L^\infty(\Omega)},$
 \item[(iii)]
	 $\|vdd^cu_1\wedge ... \wedge dd^cu_k\wedge\beta^{n-m}\|_{K}
		 \leq C~\|v\|_{ L^1(\Omega)}.\|u_1\|_{ L^\infty(\Omega)} ... \|u_k\|_{ L^\infty(\Omega)}.$
\end{enumerate}
\end{proposition}

\begin{proof}
{\bf (i)} By induction we only need to prove that
$$
	\|dd^cu_1\wedge ... \wedge dd^cu_k\wedge\beta^{n-m}\|_{K}
		\leq C~\|u_1\|_{ L^\infty(U)} \|dd^cu_2\wedge ...
			\wedge dd^cu_k\wedge\beta^{n-m}\|_U.
$$
In fact, let $\chi\geq 0$ be a test function equal to 1 on $K$.
Then an integration by parts yields
$$
	\|dd^cu_1\wedge ... \wedge dd^cu_k\wedge\beta^{n-m}\|_{K}
		\leq C\int_U\chi dd^cu_1\wedge ... \wedge dd^cu_k\wedge\beta^{n-k}
		=C\int_Uu_1dd^c\chi\wedge ... \wedge dd^cu_k\wedge\beta^{n-k}.
$$
Thus,
$$
	\|dd^cu_1\wedge ... \wedge dd^cu_k\wedge\beta^{n-m}\|_{K}
		\leq C' \|u_1\|_{ L^\infty(U)}\|dd^cu_2\wedge ...
			\wedge dd^cu_k\wedge\beta^{n-m}\|_U,
$$
where $C'$  depends only  on bounds of coefficients of $dd^c\chi$ and on the set  $U$.

{\bf (ii)} It is a simple consequence of  (i), and the result
$\| dd^c w \wedge\beta^{n-1}\|_K\leq C_{K,U}\|w\|_{ L^1(U)}$
for every $w\in SH_m(\Omega)$ (see \cite{D2}, Remark 3.4).

{\bf (iii)}  See \cite{D2} Proposition 3.11.
\end{proof}

\subsection{ $m$-pseudoconvex domains}
\label{pr3}
Let $\Omega$ be a bounded domain with $\partial\Omega$ in the class $C^2$.
Let $\rho\in C^2$ in a neighborhood of $\bar{\Omega}$
be a defining function of $\Omega$, i.e. a function such that
$$
	\rho<0 \;\; \text{on} \;\; \Omega, \;\;\;\; \rho
		= 0 \;\; \text{and} \;\; d\rho\ne 0
			\;\; \text{ on } \;\; \partial\Omega.
$$
\begin{definition}
\label{pr-de-4}
A $C^2$ bounded domain is called strongly $m$-pseudoconvex
if there is a defining function $\rho$ and some $\varepsilon>0$
such that $(dd^c\rho)^k\wedge\beta^{n-k}\geq \varepsilon\beta^n$
in $\bar{\Omega}$ for every $1\leq k\leq m$.
\end{definition}

It is obvious that a strongly pseudoconvex domain is a strongly $m$-pseudoconvex domain.
The properties of strongly $m$-pseudoconvex domains  are similar
to those of strongly pseudoconvex domains, e.g, it can be
shown that strongly $m$-pseudoconvexity is characterized by
a condition on its boundary  (see \cite{Li}, Theorem 3.1).
We also have the criterion that if the Levi form of $\Omega$
corresponding to $\rho$ belongs to
the interior of $\Gamma_{m-1}$ then $\Omega$
is strongly $m$-pseudoconvex (see \cite{Li}, Proposition 3.3).

\subsection{Cegrell's inequalities for the complex Hessian operator}
\label{pr4}
It is sufficient for our purpose in this section to work within the class of
$m$-sh functions which are continuous
up to the boundary and equal to 0 on the boundary.
Let $\Omega$ be a strongly $m$-pseudoconvex domain in $\bC^n$,
we denote
$$
	\cE_0(m)
		=\lbrace u\in SH_m(\Omega) \cap C(\bar\Omega); ~u_{|_{\partial\Omega}}=0,
	~ \int_\Omega(dd^cu)^m\wedge\beta^{n-m}<+\infty \rbrace.
$$
For the case $m=n$, this class was introduced by Cegrell in \cite{Ce1}.
It is a convex cone for $1\leq m\leq n$
(see \cite{Ce1}, p. 188 ). Our goal is to establish inequalities
very similar to the one due to Cegrell (see \cite{Ce2}, Lemma 5.4, Theorem 5.5)
for the Monge-Amp\`ere operator. In order to avoid confusions and
trivial statements we only consider $2\leq m\leq n-1.$
\begin{proposition}
\label{pr-pr-5}
Let $u,v,h\in\cE_0(m)$, and $1\leq p, q\leq m$, $p+q\leq m$, set $T=-hS$ where
$S=dd^ch_1\wedge...\wedge dd^ch_{m-p-q}\wedge\beta^{n-m}$
with $h_1,...,h_{m-p-q}$ are also in $\cE_0(m)$, then
$$
	 \int_\Omega (dd^cu)^p\wedge (dd^cv)^q\wedge T
		\leq\left[\int_\Omega
			(dd^cu)^{p+q}\wedge T\right]^\frac{p}{p+q}
				\left[\int_\Omega (dd^cv)^{p+q}\wedge T\right]^\frac{q}{p+q}.
$$
\end{proposition}

\begin{proof}
See Lemma 5.4 in \cite{Ce2}.
We only remark here that two sides of the inequality are finite because of the
convexity of the cone $\cE_0(m)$.
\end{proof}

\begin{remark}
\label{pr-repr5}
The statement in Proposition 1.5 is still true when
$h\in SH_m\cap  L^\infty(\Omega)$,
$\lim_{\zeta \rightarrow \partial\Omega} h(\zeta)=0$ and
$\int_\Omega (dd^ch)^m\wedge \beta^{n-m}<+\infty$ since the integration by
 parts formula is valid as in the case of the continuous case (see \cite{Ce2}, Corollary 3.4 ).
\end{remark}

Applying Proposition~\ref{pr-pr-5} for some special cases of
$m$-sh functions in $\cE_0(m)$ we obtain

\begin{corollary}
\label{pr-co-6}
For $u,v,h\in\cE_0(m)$, $1\leq p\leq m-1$, then
\begin{enumerate}
\item[(i)]
\begin{equation*}
\begin{aligned}
\int_\Omega -h  (dd^cu)^p\wedge (dd^cv)^{m-p}\wedge & \beta^{n-m} \\
	& \leq\left[\int_\Omega
	- h(dd^cu)^m\wedge\beta^{n-m}\right]^\frac{p}{m}
	\left[\int_\Omega -h(dd^cv)^m\wedge\beta^{n-m}\right]^\frac{m-p}{m},
\end{aligned}
\end{equation*}
\item[(ii)]
$\int_\Omega(dd^cu)^p\wedge (dd^cv)^{m-p}\wedge\beta^{n-m}
	\leq\left[\int_\Omega (dd^cu)^m\wedge\beta^{n-m}\right]^\frac{p}{m}
	\left[\int_\Omega(dd^cv)^m\wedge\beta^{n-m}\right]^\frac{m-p}{m}.$
\end{enumerate}
\end{corollary}

\begin{proof}
{\bf (i)}  follows from Proposition~\ref{pr-pr-5} when $u=u_1=...=u_p$, $v=v_1=...=v_q$.
{\bf (ii)} comes from the fact that for
$\rho$ a defining function of $\Omega$ we have
$$
	\int_\Omega(dd^cu)^p\wedge (dd^cv)^{m-p}\wedge\beta^{n-m}
		=\lim_{\varepsilon\rightarrow 0}\int_{\{\rho< -\varepsilon\}}
			(dd^cu)^p\wedge (dd^cv)^{m-p}\wedge\beta^{n-m},
$$
and
\begin{align*}
& \int_{U_\varepsilon}(dd^cu)^p\wedge (dd^cv)^{m-p}\wedge\beta^{n-m} \\
	&\leq \int_\Omega -h^*_{U_\varepsilon,\Omega}(dd^cu)^p
		\wedge (dd^cv)^{m-p}\wedge\beta^{n-m}\\
	&\leq\left[\int_\Omega -h^*_{U_\varepsilon,\Omega}(dd^cu)^m
		\wedge\beta^{n-m}\right]^\frac{p}{m}
	\left[ \int_\Omega -h^*_{U_\varepsilon,\Omega}(dd^cv)^m
		\wedge\beta^{n-m}\right]^\frac{m-p}{m}\\
	&\leq \left[\int_\Omega (dd^cu)^m\wedge\beta^{n-m}\right]^\frac{p}{m}
		\left[\int_\Omega (dd^cv)^m\wedge\beta^{n-m}\right]^\frac{m-p}{m},
\end{align*}
where $U_\varepsilon=\{\rho<-\varepsilon\}$ and
$ h_{U_\varepsilon,\Omega}
	= \sup\{u\in SH_m(\Omega);~ u\leq 0;~ u_{|_{U_\varepsilon}}\leq -1\}$.
It is clear that $-1\leq h^*_{U_\varepsilon,\Omega}\leq 0$,
$\lim_{\zeta\rightarrow\partial\Omega}h^*_{U_\varepsilon,\Omega}(\zeta)=0$ and
$\int_\Omega (dd^c h^*_{ U_\varepsilon,\Omega})^m\wedge \beta^{n-m}<+\infty$.
Hence the inequality (i) is still applicable by Remark~\ref{pr-repr5}.
\end{proof}

\subsection{ $m$-capacity, convergence theorems, the comparison principle}
\label{pr}
For $E$ a Borel set in $\Omega$ we define
$$
	cap_m(E,\Omega)
		= \sup \{ \int_E (dd^cu)^m\wedge \beta^{n-m},
			\; u\in SH_m(\Omega), \; 0\leq u\leq 1\}.
$$
In view of Proposition~\ref{cln}, it is finite as soon as
$E$ is relatively compact in $\Omega$. 
This is the version of the relative capacity in the case of
$m$-subharmonic functions. It is an useful tool to
establish convergent properties, especially the comparison principle.






\begin{theorem}[Convergence theorem]
\label{pr-th-9}
Let $\{u_k^j\}_{j=1}^{\infty}$, $k=1,...,m$ be locally uniformly bounded sequences of
$m$-subharmonic functions in $\Omega$,
$u_k^j\rightarrow u_k\in SH_m\cap  L^\infty(\Omega)$ in $\c_m$ as $j\rightarrow \infty$.
Then
$$
	\lim_{j\rightarrow\infty} dd^cu_1^j\wedge ....\wedge dd^cu_m^j\wedge\beta^{n-m}
		= dd^cu_1\wedge ...\wedge dd^c u_m\wedge\beta^{n-m}
$$
in the topology of currents.
\end{theorem}

\begin{proof}
See the proof of Theorem 1.11 in \cite{K2}.
\end{proof}

\begin{remark}
\label{pr-reth9}
One may prove as in Theorem 2.1 of \cite{BT2} that for $1\leq k\leq m$, let $u_1^j,...,u_k^j$
be decreasing sequences of locally bounded $m$-sh functions such that
$\lim_{j\rightarrow\infty}u^j_l(z)=u_l(z)\in SH_m\cap L^\infty_{loc}(\Omega)$
for all $z\in\Omega$ and $1\leq l\leq k$. Then
$$
	\lim_{j\rightarrow\infty} dd^cu_1^j\wedge ....\wedge dd^cu_k^j\wedge\beta^{n-m}
		= dd^cu_1\wedge ...\wedge dd^c u_k\wedge\beta^{n-m}
$$
in the sense of currents. Thus, the currents  obtained in the inductive definition
$(H_k)$ of the wedge product of currents associated to locally bounded
$m$-sh functions are closed positive currents.
\end{remark}

\begin{proposition}
\label{pr-pr-10}
If  $u_j\in  SH_m\cap  L^\infty(\Omega)$ is a sequence decreasing to a  bounded function
 $u$ in $\Omega$ then it converges to $u\in SH_m\cap  L^\infty(\Omega)$ with respect to
 $\c_m$. In particular, Theorem~\ref{pr-th-9} holds in this case.
\end{proposition}

\begin{proof}
See Proposition 1.12 in \cite{K2}.
\end{proof}

\begin{theorem}[Quasi-continuity]
\label{pr-th-11}
For a $m$-subharmonic function $u$ defined in $\Omega$ and for each
$\varepsilon>0$, there is an open subset $U$ such that
$\c_m(U,\Omega) < \varepsilon$ and $u$ is continuous in $\Omega \setminus U$.
\end{theorem}

\begin{proof}
See Theorem 1.13 in \cite{K2}.
\end{proof}

From the quasi-continuity of $m$-subharmonic functions
one can derive several important results.
\begin{theorem}
\label{pr-th-12}
Let $u,v$ be locally bounded $m$-sh functions on $\Omega$.
Then we have an inequality of measures
$$
	(dd^c\max\{u,v\})^m\wedge\beta^{n-m}
		\geq {\bf 1}_{\{u\geq v\}}(dd^cu)^m\wedge\beta^{n-m}
			+{\bf 1}_{\{u<v\}}(dd^cv)^m\wedge\beta^{n-m}.
$$
\end{theorem}

\begin{proof}
See Theorem 6.11 in \cite{D1}.
\end{proof}

\begin{theorem}[Comparison principle]
\label{pr-th-13}
Let $\Omega$ be an open bounded subset of $\bC^n$.
For $u,v\in SH_m\cap  L^\infty(\Omega)$ satisfying
$\liminf_{\zeta\rightarrow z}(u-v)(\zeta)\geq 0$ for any $z\in \partial \Omega$, we have
$$
	 \int_{\{u<v\}} (dd^cv)^m\wedge\beta^{n-m}
	 	 \leq \int_{\{u<v\}} (dd^cu)^m\wedge\beta^{n-m}.
$$
\end{theorem}

\begin{proof}
The proof follows the  lines of the proof of  Theorem 1.16 in  \cite{K2}.
First consider $u,v\in C^\infty(\Omega)$,
$E=\{u<v\}\subset\subset\Omega$, and smooth $\partial\Omega$.
In this case, put $u_\varepsilon=\max\{u+\varepsilon,v\}$ and use Stokes' theorem to get
\begin{equation}
\label{pr-th13-0}
\begin{aligned}
& \int_E(dd^cu_\varepsilon)^m\wedge\beta^{n-m}
	=\int_{\partial E} d^cu_\varepsilon\wedge (dd^c u_\varepsilon)^{m-1}\wedge\beta^{n-m} \\
	&= \int_{\partial E} d^cu\wedge (dd^c u)^{m-1}\wedge\beta^{n-m}
		=\int_{ E} (dd^c u)^{m}\wedge\beta^{n-m}
\end{aligned}
\end{equation}
(since $u_\varepsilon=u+\varepsilon$ on neighborhood of $\partial E$).
By Theorem~\ref{pr-th-9}, $(dd^cu_\varepsilon)^m\wedge\beta^{n-m}$
converges weakly$^*$  to $(dd^cv)^m\wedge\beta^{n-m}$ as
$\varepsilon\rightarrow 0$ on the open set $E$, it implies that
$$
	\int_E (dd^cv)^m\wedge\beta^{n-m}
		\leq \liminf_{\varepsilon\rightarrow \infty}
			\int_E(dd^cu_\varepsilon)^m\wedge\beta^{n-m}.
$$
This combining with \eqref{pr-th13-0} imply the statement.

For the general case, suppose $\|u\|, \|v\|<1$, fix $\varepsilon>0$ and $\delta>0$.
From the quasi-continuity, there is an open set $U$ such that
$cap_m(U,\Omega)<\varepsilon$ and $u=\tilde{u}$, $v=\tilde v$ on
$\Omega\setminus U$ for some continuous functions
$\tilde u$, $\tilde v$ in $\Omega$. Let $u_k$, $v_k$
be the standard regularizations of $u$ and $v$.
By Dini's theorem $u_k$ and $v_k$ uniformly converge  (correspondingly) to
$u$ and to $v$ on $\Omega\setminus U$. Then for $k>k_0$ big enough, subsets
$E(\delta):=\{\tilde u+\delta<\tilde v\}$ and $E_k(\delta):=\{u_k+\delta<v_k\}$ satisfy
\begin{equation}
\label{pr-th13-1}
 E(2\delta)\setminus U\subset\subset \bigcap_k E_k(\delta)\setminus U
	\;\; \text{ and } \;\; \bigcup_k E_k(\delta)\setminus U
			\subset\subset\{\tilde u<\tilde v\}.
\end{equation}
In what follows we shall often use the estimate
$$
	\int_U (dd^cw)^m\wedge\beta^{n-m}
		\leq cap_m(U,\Omega)<\varepsilon
		\;\; \text{ where } \;\; 0\leq w\leq 1,
$$
not mentioning this any more.
Since $\{u+2\delta<v\}=\{\tilde u+2\delta<\tilde v\}$ on $\Omega\setminus U$ ,
\begin{equation}
\label{pr-th13-2}
\begin{aligned}
 \int_{\{u+2\delta<v\}}(dd^cv)^m \wedge\beta^{n-m}
	&\leq\int_{\{\tilde u+2\delta<\tilde v\}\setminus U}
		(dd^cv)^m\wedge\beta^{n-m} +\varepsilon \\
	& = \int_{E(2\delta)\setminus U}
		(dd^cv)^m\wedge\beta^{n-m} +\varepsilon.
\end{aligned}
\end{equation}
Since $(dd^cv_k)^m\wedge\beta^{n-m}$ weakly$^*$ converges to
$(dd^cv)^m\wedge\beta^{n-m}$ and
$E(2\delta)$ is open and by \eqref{pr-th13-1} we get
\begin{equation}
\label{pr-th13-3}
\begin{aligned}
\int_{E(2\delta)}(dd^cv)^m\wedge\beta^{n-m}
	& \leq \liminf_{k\rightarrow\infty}\int_{E(2\delta)}
		(dd^cv_k)^m\wedge\beta^{n-m} \\
	& \leq \liminf_{k\rightarrow\infty}\int_{E_k(\delta)}
		(dd^cv_k)^m\wedge\beta^{n-m}+\varepsilon.
\end{aligned}
\end{equation}
Now, from Sard's theorem, we may assume that $E_k(\delta)$ has smooth boundary
(changing $\delta$ if needed), thus using the argument of the smooth case we have
\begin{equation}
\label{pr-th13-4}
\int_{E_k(\delta)}(dd^cv_k)^m\wedge\beta^{n-m}
	\leq \int_{E_k(\delta)}(dd^cu_k)^m\wedge\beta^{n-m}.
\end{equation}
Therefore, by \eqref{pr-th13-2}, \eqref{pr-th13-3} and \eqref{pr-th13-4}, we have
\begin{equation}
\label{pr-th13-5}
 \int_{\{u+2\delta<v\}}(dd^cv)^m\wedge\beta^{n-m}
 	\leq  \liminf_{k\rightarrow\infty}\int_{E_k(\delta)}
		(dd^cu_k)^m\wedge\beta^{n-m}+2\varepsilon.
\end{equation}
Furthermore, using \eqref{pr-th13-1} and the fact that
$(dd^cu_k)^m\wedge\beta^{n-m}$ weakly$^*$  converges to
$(dd^cu)^m\wedge\beta^{n-m}$ we obtain
\begin{equation}
\label{pr-th13-6}
\limsup_{k\rightarrow\infty}\int_{\overline{\cup_k E_k(\delta)\setminus U}}
	(dd^cu_k)^m\wedge\beta^{n-m}
\leq \int_{\overline{\cup_k E_k(\delta)\setminus U}}
		(dd^cu)^m\wedge\beta^{n-m}.
\end{equation}
Thus, from \eqref{pr-th13-1}, \eqref{pr-th13-5} and \eqref{pr-th13-6} one has
\begin{equation}
\label{pr-th13-7}
 \int_{\{u+2\delta<v\}}(dd^cv)^m\wedge\beta^{n-m}
 	\leq \int_{\{\tilde u<\tilde v\}}(dd^cu)^m\wedge\beta^{n-m}+3\varepsilon
	\leq \int_{\{u<v\}}(dd^cu)^m\wedge\beta^{n-m}+4\varepsilon.
\end{equation}
Finally, letting $\delta$ and $\varepsilon$ tend to $0$ in \eqref{pr-th13-7} the statement is proved.
\end{proof}

\begin{corollary}
\label{pr-co-14}
Under the assumption of Theorem~\ref{pr-th-13} we have
\begin{enumerate}
\item[(a)]
If $(dd^cu)^m\wedge\beta^{n-m}\leq (dd^cv)^m\wedge\beta^{n-m}$ then $v\leq u$,
\item[(b)]
If $(dd^cu)^m\wedge\beta^{n-m} = (dd^cv)^m\wedge\beta^{n-m}$
and $\lim_{\zeta\rightarrow z}(u-v)(\zeta)=0$ for $z\in\partial\Omega$ then $u=v$,
\item[(c)]
If $\lim_{\zeta\rightarrow \partial \Omega} u(\zeta)
	=\lim_{\zeta\rightarrow \partial \Omega} v(\zeta)=0$
and $u\leq v$ in $\Omega$, then
$$
	\int_\Omega (dd^c v)^m\wedge \beta^{n-m}
		\leq \int_\Omega (dd^cu)^m\wedge \beta^{n-m}.
$$
\end{enumerate}
\end{corollary}

\begin{proof}
For (a) and (b) see Corollary 1.17 in \cite{K2}. For (c), let $\varepsilon>0$,
applying Theorem~\ref{pr-th-13} we have
$$
	 \int_\Omega (dd^c v)^m\wedge \beta^{n-m}
		\leq (1+\varepsilon)^n\int_\Omega (dd^cu)^m\wedge \beta^{n-m}.
$$
Then, letting $\varepsilon \rightarrow 0$ which gives the result.
\end{proof}

\bigskip

\section{Subsolution theorem}
\label{su}
\setcounter{equation}{0}
In this section we will prove our main theorem. The method we use here is similar to the
one from the proof of the  plurisubharmonic case (see \cite{K2}, Theorem 4.7). We first recall
the theorem due to Dinew and Ko\l odziej about the weak solution of the complex Hessian
equation with the  right hand side in $ L^p$ (see \cite{DK}, Theorem 2.10). From now on we
only consider $1<m<n$.
\begin{theorem}[\cite{DK}]
\label{su-th-0}
Let $\Omega$ be a  smoothly strongly $m$-pseudoconvex domain. Then for $p>n/m$,
$f\in  L^p(\Omega)$ and a continuous function $\varphi$ on $\partial\Omega$ there exists
$u\in SH_m(\Omega)\cap C(\bar\Omega )$ satisfying
$$
	(dd^cu)^m\wedge\beta^{n-m}=f\beta^n ,
$$
and $u=\varphi$ on $\partial\Omega$.
\end{theorem}

Let us state the subsolution theorem

\begin{theorem}
\label{su-th-1}
Let $\Omega$ be a smoothly strongly $m$-pseudoconvex domain in $\bC^n$,
and let $\mu$ be a finite  positive Borel measure in $\Omega$.
If there is a subsolution $v$, i.e
\begin{equation}
\label{su-th1-1}
\begin{cases}
	v \in  SH_m \cap  L^{\infty}(\Omega), \\
	(dd^cv)^m \wedge \beta^{n-m} \geq d\mu, \\
	\lim_{\zeta\rightarrow z} v(\zeta)=\varphi (z)
		\text{  for any  }   z\in\partial\Omega,
\end{cases}
\end{equation}
then there is a solution $u$ of the following Dirichlet problem
\begin{equation}
\label{su-th1-2}
 \begin{cases}
	u \in  SH_m\cap  L^{\infty}(\Omega), \\
	(dd^cu)^m \wedge \beta^{n-m} = d\mu, \\
	 \lim_{\zeta\rightarrow z} u(\zeta)
	 	=\varphi (z)  \text{  for any  }   z\in\partial\Omega.
\end{cases}
\end{equation}
\end{theorem}

\begin{proof}
We first prove Theorem 2.1 under two extra assumptions:

1) the measure $\mu$ has compact support in $\Omega$;

2) the function $\varphi$ is in the class $C^2 .$

 Using the first of those conditions we  can modify $v$ so that
 $v$ is $m$-subharmonic in a neighborhood of $\Omega$ (and still is a subsolution).
 To do this take an open subset
 $supp~\mu \subset \subset U \subset \subset\Omega$ and consider the envelope
$$
	 \hat{v}= \sup\lbrace w\in  SH_m(\Omega): w\leq 0,~w
	 	\leq v ~\text{ on }~ U  \rbrace.
$$
Then from Proposition~\ref{pr-pr-3}-\eqref{pr-pr3-6} $\hat{v}^*$ is  a competitor in the definition of
the envelope, hence $\hat{v}=\hat{v}^*\in SH_m(\Omega)$. The balayage procedure implies that
$\hat{v}=v$ on $U$ and $\lim_{\zeta\rightarrow z} \hat{v}(\zeta)=0$ for any
$z\in\partial\Omega$ (the balayage  still works
as in the case of plurisubharmonic functions  by results in \cite{Bl}, Theorem 1.2, Theorem 3.7).
Thus,
$(dd^c\hat{v})^m\wedge\beta^{n-m}\geq d\mu$ as $supp~ \mu \subset \subset U$.
Next, take $\rho$ a defining  function of $\Omega$ which is smooth on a neighborhood
$\Omega_1$ of $\bar{\Omega}$ and $(dd^c\rho)^k\wedge\beta^{n-k}\geq \varepsilon\beta^n$,
 $1\leq k\leq m$, in $\bar{\Omega}$ for some
 $\varepsilon>0$. Since $\hat{v}$ is bounded we can further choose
 $\rho$ satisfying $\rho\leq\hat{v}$ on $\bar{U}$. Put
$$
v_1(z):=
\begin{cases}
	\max\{\rho(z),\hat{v}(z)\} &\text{on} \;\; \bar{\Omega}, \\
	\rho(z)			&\text{on} \;\; \Omega_1\setminus \bar{\Omega}.
\end{cases}
$$
Hence $v_1$ is  a subsolution which is defined and $m$-subharmonic
in a neighborhood of $\bar{\Omega }$.
We  still write  $v$ instead of $v_1$ in what follows. Furthermore,
 using the balayage procedure (as above)
one can make the support of $d\nu:=(dd^cv)^m\wedge\beta^{n-m}$ compact in $\Omega$.

Now, we can sketch the rest of the proof of the theorem.
We will approximate $d\mu$ by a sequence of
measures $\mu_j$ for which the Dirichlet problem is solvable
(using Theorem~\ref{su-th-0})  obtaining a
sequence of solutions $\{u_j\}$ corresponding to $\mu_j$.
Then we take a limit point $u$ of $\{u_j\}$ in
$ L^1(\Omega)$. Finally we show that $u_j\rightarrow u$ with respect to $\c_m$ in order to
conclude that $u$ is a  solution of \eqref{su-th1-2}.

By the Radon-Nikodym theorem $d\mu=hd\nu$, $0\leq h\leq 1$.
For the subsolution $v$ one can define the
 regularizing sequence $w_j\downarrow v$ in a neighborhood of the closure of $\Omega$.
 Let us write
 $(dd^cw_j)^m\wedge\beta^{n-m}=g_j\beta^n$, $\mu_j:=hg_j\beta^n$.
 Then by Proposition~\ref{pr-pr-10}
 $\lim_{j\rightarrow\infty} \mu_j=\mu$. As $\mu$ has compact support,
 so  $\mu_j$'s does. In particular,
 $hg_j\in  L^p(\Omega)$ for every  $p>0$.
 Therefore, applying Theorem~\ref{su-th-0} we have $u_j$ solving
\begin{equation}
\label{su-pf-1}
\begin{cases}
	u_j\in  SH_m(\Omega) \cap C(\bar{\Omega}), \\
	(dd^cu_j)^m\wedge\beta^{n-m}=\mu_j,\\
  	u_j (z)=\varphi (z)  \text{  for   }   z\in\partial\Omega.
\end{cases}
\end{equation}
Now we set $u=(\limsup u_j)^*$, and passing to a subsequence we assume that
$u_j$ converges to $u$ in $ L^1(\Omega)$.
From the definition of $w_j$ they are uniformly bounded.
Choosing a uniform constant $C$such that $w_j-C < \varphi$
on $\partial \Omega$, by Corollary~\ref{pr-co-14}-(a),
$w_j-C\leq u_j\leq\sup_{\bar\Omega}\varphi$. Thus, $\{u_j\} $ is uniformly bounded.
In particular, $u$ is also
bounded and now we shall check that
$\lim_{\Omega\ni\zeta\rightarrow z}u(\zeta)=\varphi(z)$ for every
$z\in\partial\Omega$. For this we only need $\varphi$ to be continuous.

Since $w_j$ converges uniformly to $v$ on $\partial\Omega$ and
$\partial\Omega$ is compact, given $\varepsilon>0$
we have  $|w_j-v|<\varepsilon$ on a small neighborhood of $\partial\Omega$
when $j$ big enough. Since $\varphi$ is continuous on $\partial\Omega$,
there is an approximant $g\in C^2(\overline{\Omega})$ of the continuous extension of
$\varphi$ such that $|g-\varphi|<\varepsilon$ on $\partial\Omega$.
For $A>0$ big enough, $A\rho+g$ is a  $m$-sh function.
By the comparison principle,
it implies that $w_j+A\rho+g-2\varepsilon\leq u_j$ on $\Omega$.
Then  $v+A\rho+\varphi-4\varepsilon\leq u_j$ on a small neighborhood of
$\partial\Omega$ for $j$ big enough.
Hence, $v+A\rho+\varphi-4\varepsilon\leq \liminf_{j\rightarrow \infty} u_j\leq u$
on a small neighborhood of $\partial\Omega$.
Because this is true  for  arbitrary $\varepsilon>0$, we obtain
 $\lim_{\zeta\rightarrow z}u(\zeta)=\varphi(z)$ for any $z\in\partial\Omega.$

The difficult part is to show that $u_j$ converges in $\c_m$ to $u$.
\begin{lemma}
\label{su-le-2}
The function $u$ defined above solves the Dirichlet problem \eqref{heq}
provided that for any $a>0$ and any compact $K\subset\Omega$ we have
\begin{equation}
\label{su-le2-1}
\lim_{j\rightarrow\infty} \int_{K\cap\{u-u_j\geq a\}}
		(dd^cu_j)^m\wedge\beta^{n-m}
		 = \lim_{j\rightarrow\infty}\mu_j(K\cap\{u-u_j\geq a\})
		 =0.
\end{equation}
\end{lemma}

\begin{proof}[Proof of Lemma~\ref{su-le-2}]
Using Theorem~\ref{pr-th-12}  we have
\begin{align*}
(dd^cu_j)^m \wedge \beta^{n-m}
	& = 1_{\{u-u_j\geq  a\}}(dd^cu_j)^m\wedge\beta^{n-m}
		+1_{\{u-u_j<a\}}(dd^cu_j)^m\wedge\beta^{n-m} \\
	&\leq  1_{\{u-u_j\geq a\}}\mu_j
		+ (dd^c\max\{u,u_j+a\})^m\wedge\beta^{n-m}.
\end{align*}
It follows that
\begin{equation}
\label{su-le2-2}
	\mu_j
		\leq 1_{\{u-u_j\geq a\}}\mu_j+(dd^c\max\{u-a,u_j\})^m\wedge\beta^{n-m}.
\end{equation}
Now, for any integer $s$ we may choose $j(s)$ such that $\mu_{j(s)}(\{u-u_{j(s)}\geq 1/s\})<1/s$.
From \eqref{su-le2-1} and \eqref{su-le2-2} we infer that
\begin{equation}
\label{su-le2-2'}
	\mu
		\leq \liminf_{s\rightarrow\infty} (dd^c\rho_s)^m\wedge\beta^{n-m},
\end{equation}
it means that $ \mu$ is less than any limit point of the right hand side,
where $\rho_s=\max\{u-1/s, u_{j(s)}\}$.
By the Hartogs lemma, $\rho_s\rightarrow u$ uniformly on any compact $E$
such that $u_{|_E}$ is continuous. So it follows from the quasi-continuity of
 $m$-sh functions that $\rho_s$ converges to $u$ in $\c_m$.
 Therefore, by Theorem~\ref{pr-th-9}
$(dd^c\rho_s)^m\wedge\beta^{n-m}\rightarrow (dd^cu)^m\wedge\beta^{n-m}$ as measures.
This combined with \eqref{su-le2-2'} implies
\begin{equation}
\label{su-le2-3}
	\mu\leq (dd^cu)^m\wedge\beta^{n-m}.
\end{equation}
For the reverse inequality, let
$\Omega_\varepsilon = \{z\in\Omega~; dist(z,\partial\Omega)<\varepsilon\}$.
We will show that for $\varepsilon>0$
\begin{equation}
\label{su-le2-3'}
\mu(\Omega)\geq \int_{\Omega_\varepsilon} (dd^cu)^m\wedge\beta^{n-m}.
\end{equation}
Indeed, firstly we note that $\rho_s=u_{j(s)}$ on a neighborhood of
$\partial\Omega_\varepsilon$ for $\varepsilon$ small enough
since $u-u_{j(s)}<1/s$ on $\partial\Omega$, $u-u_{j(s)}$ is upper semi-continuous
on $\Omega$ and $\partial\Omega$ is compact.
Hence, by the weak$^*$ convergence  $\mu_{j(s)}\rightarrow \mu$  and Stokes' theorem,
\begin{align*}
 \mu(\Omega)\geq \mu(\overline{\Omega_\varepsilon})
 	&\geq \limsup_{j(s)\rightarrow\infty} \mu_{j(s)}(\overline{\Omega_\varepsilon}) \\
 	&\geq \liminf_{j(s)\rightarrow\infty}\mu_{j(s)}(\Omega_\varepsilon) \\
 	&= \liminf_{j(s)\rightarrow\infty} \int_{\Omega_\varepsilon}
		(dd^cu_{j(s)})^m\wedge\beta^{n-m}
	  =\liminf_{j(s)\rightarrow\infty} \int_{\Omega_\varepsilon}
	  	(dd^c\rho_s)^m\wedge\beta^{n-m} \\
 	&\geq  \int_{\Omega_\varepsilon} (dd^cu)^m\wedge\beta^{n-m},
 \end{align*}
where in the last inequality we use  the weak$^*$ convergence
$(dd^c\rho_s)^m\wedge\beta^{n-m}\rightarrow (dd^cu)^m\wedge\beta^{n-m}$.
Therefore, \eqref{su-le2-3'}is proved. Let $\varepsilon\rightarrow 0$, then it implies
$\mu(\Omega)\geq (dd^cu)^m\wedge\beta^{n-m}(\Omega)$.
Thus the measures in \eqref{su-le2-3} are equal. The lemma follows.
\end{proof}

It remains to prove \eqref{su-le2-1} in Lemma 2.2 above.
It is a consequence of the following lemma.
\begin{lemma}
\label{su-le-3}
Suppose that there is a subsequence of $\{u_j\}$, still denoted by $\{u_j\}$, such that
\begin{equation*}
\label{su-le3-1}
	 \int_{\{u-u_j\geq a_0\}} (dd^cu_j)^m\wedge\beta^{n-m}>A_0, \;\; A_0>0, a_0>0.
\end{equation*}
Then, for $0\leq p \leq m$ there exist $a_p$, $A_p$, $k_1>0$ such that
\begin{equation}
\label{su-le3-1'}
	\int_{\{u-u_j
		\geq a_p\}} (dd^cv_j)^{m-p}\wedge (dd^cv_k)^{p}\wedge\beta^{n-m}>A_p,
		\;\;  j>k>k_1 ,
\end{equation}
for $v_j's$ the solutions  (from Theorem~\ref{su-th-0}) of the  Dirichlet problem
\begin{equation}
\label{su-le3-2}
\begin{cases}
	v_j\in  SH_m(\Omega)\cap C(\bar{\Omega}), \\
	(dd^cv_j)^m\wedge\beta^{n-m} = \nu_j \; (= g_j\beta^n),\\
	v_j(z)= 0 \;\; \text{ on } \;\; \partial\Omega.
\end{cases}
\end{equation}
Note that $\{v_j\}$ is uniformly bounded as a consequence of the uniform boundedness of
$\{w_j\}$ and Corollary~\ref{pr-co-14}-(a).
\end{lemma}

\begin{proof}[Proof of Lemma~\ref{su-le-3}]
We will prove it  by induction over $p$. For $p=0$ the statement holds by the hypothesis.
Suppose that \eqref{su-le3-1'} is true for $p<m$, we need to prove it for $p+1$.
The first observation is that if
$T(r,s):=(dd^cu_r)^q\wedge (dd^cv_s)^{m-q}\wedge\beta^{n-m}$
then there is a constant $C$ independent of $r,s$ such that
\begin{equation}
\label{su-le3-3}
\int_\Omega T(r,s)\leq C.
\end{equation}

Indeed, fix  a $C^2$ extension of $\varphi$ to a neighborhood of the closure of
$\Omega $. If $\rho$ is a defining function of $\Omega$, then there is a constant
$A>0$ such that $A\rho\pm\varphi\in SH_m(\Omega)$.
We shall check that $u_r+A\rho-\varphi$ belongs to $\cE_0(m)$.
It is enough to verify
$$
	\int_\Omega (dd^c(u_r+A\rho - \varphi))^m\wedge\beta^{n-m}<+\infty.
$$
In fact, from
$(dd^cu_r)^m\wedge \beta^{n-m}
	=hg_r \beta^n\leq  (dd^c (M_r\rho+\varphi))^m\wedge \beta^{n-m}$
for some $M_r>0$ and Corollary~\eqref{pr-co-14}-(a) we have
$u_r\geq M_r\rho +\varphi$ in $\Omega$.
Hence, $u_r+A\rho -\varphi\geq (M_r+A)\rho$ in $\Omega$.
Thus, by Corollary~\ref{pr-co-14}-(c)
$$
	\int_\Omega (dd^cu_r+A\rho-\varphi)^m\wedge \beta^{n-m}
		\leq \int_\Omega (dd^c (M_r+A)\rho)^m\wedge \beta^{n-m}<+\infty.
$$
Now, we note that $\mu_r(\Omega)$ and $\nu_s(\Omega)$ are bounded as
$\mu$ and $\nu$ have compact support.
Next, from Cegrell's inequalities, Corollary~\ref{pr-co-6}-(ii), for $1\leq k\leq m-1$, it implies
\begin{equation*}
\label{su-le3-3.1}
\begin{aligned}
\int_\Omega (dd^c(u_r+A\rho-\varphi))^k
	& \wedge(dd^c\rho)^{m-k}\wedge\beta^{n-m} \\
	&\leq \left[\int_\Omega (dd^c(u_r+A\rho-\varphi))^m\wedge\beta^{n-m}\right]^\frac{k}{m}
		\left[\int_\Omega(dd^c\rho)^m\wedge\beta^{n-m}\right]^\frac{m-k}{m}.
\end{aligned}
\end{equation*}
Hence,
\begin{equation}
\label{su-le3-3.2}
\begin{aligned}
I(r)	=      &\, \int_\Omega (dd^c(u_r+A\rho-\varphi))^m\wedge\beta^{n-m}\\
	\leq &\,  \int_\Omega (dd^cu_r)^m\wedge\beta^{n-m}
		+\int_\Omega (dd^c(A\rho-\varphi))^m\wedge\beta^{n-m}\\
	\;\; &\,+ C(A,\varphi)\sum_{k=1}^{m-1}\int_\Omega (dd^cu_r
		+A\rho -\varphi)^k\wedge(dd^c\rho)^{m-k}\wedge\beta^{n-m} \\
	\leq &\, \mu_r(\Omega)+C(A,\rho,\varphi)\\
	\;\; &\, +C(A,\varphi)\sum_{k=1}^{m-1}
			\left[\int_\Omega(dd^c(u_r	 + A\rho-\varphi))^m\wedge\beta^{n-m} \right]^\frac{k}{m}
		\left[\int_\Omega(dd^c\rho)^m\wedge\beta^{n-m} \right]^\frac{m-k}{m} \\
	\leq &\, \mu_r(\Omega)+C(A,\rho,\varphi)
		+C'(A,\varphi,\rho) \sum_{k=1}^{m-1}\left[I(r)\right]^\frac{k}{m}.
\end{aligned}
\end{equation}
Consider the two sides of the inequality \eqref{su-le3-3.2} as two positive functions in $r$.
$\mu_r(\Omega)$'s are bounded, and the degree of $I(r)$ on the right hand side is strictly
less than the degree of $I(r)$ on the left hand side, therefore $I(r)$ are bounded by
a constant independent of $r$. Again, by Cegrell's inequalities, Corollary~\ref{pr-co-6}-(ii),
as   $v_s$ obviously  belongs to $\cE_0(m)$,
\begin{equation*}
\label{su-le3-3.3}
\begin{aligned}
\int_\Omega T(r,s)
	&\leq \int_\Omega (dd^c(u_r+A\rho-\varphi))^q\wedge(dd^cv_s)^{m-q}\wedge\beta^{n-m}\\
	&\leq \left[\int_\Omega (dd^c(u_r+A\rho-\varphi))^m\wedge\beta^{n-m}\right]^\frac{q}{m}
		\left[ \int_\Omega (dd^cv_s)^m\wedge\beta^{n-m}\right]^\frac{m-q}{m}\\
	&\leq \left[I(r)\right]^\frac{q}{m}\left[\nu_s(\Omega)\right]^\frac{m-q}{m}\\
	&\leq C''(A,\varphi,\rho),
\end{aligned}
\end{equation*}
 because $I(r)$ and $\nu_s(\Omega)$ are bounded. Thus we have proved \eqref{su-le3-3}.
 We may assume that $-1<u_j, v_j<0$,
 because all functions $u_j,v_j$ are uniformly bounded by a constant independent of $j$,
 the estimates in the statement of Lemma~\ref{su-le-3} will only be changed by a uniformly
 positive constant. To simplify notations we set
 $S(j,k):=(dd^cv_j)^{m-p-1}\wedge (dd^cv_k)^p\wedge\beta^{n-m}$.
 Fix a positive  number $d>0$ (specified later in \eqref{su-le3-5})
 and recall that we need a uniform estimate from below for
 $\int_{\{u-u_j\geq d\}} dd^cv_k\wedge S(j,k)$.
 From the assumption on $u_j, v_j$, we have $u-u_j\leq {\bf 1}_{\{u-u_j \geq d\}}+d$.
 It follows that
\begin{equation*}
\label{su-le3-4'}
\begin{aligned}
J(j,k):=
	\int_\Omega (u-u_j)(dd^cv_k)\wedge S(j,k)
	& \leq \int_\Omega {\bf 1}_{\{u-u_j\leq d\}}dd^cv_k\wedge S(j,k)
		+ d\int_\Omega dd^cv_k\wedge S(j,k) \\
	& \leq \int_{\{u-u_j\geq d\}}dd^cv_k\wedge S(j,k) + d C,
\end{aligned}
\end{equation*}
where $C$ is from \eqref{su-le3-3}. Therefore
\begin{equation}
\label{su-le3-4}
\int_{\{u-u_j\geq d\}}dd^cv_k\wedge S(j,k)\geq J(j,k)- dC.
\end{equation}
The induction hypothesis says that there exist $a_p, A_p>0$ and $k_1>0$ such that
\begin{equation}
\label{su-le3-5}
  \int_{\{u-u_j\geq a_p\}}
  	(dd^cv_j)^{m-p}\wedge (dd^cv_k)^{p}\wedge\beta^{n-m}
  		>A_p, \;\;  j>k>k_1.
\end{equation}
We fix another small positive constant $\varepsilon>0$ and put
$J'(j,k):=\int_\Omega (u-u_j)dd^cv_j\wedge S(j,k)$.

{\it Claim.}
\begin{enumerate}
\item[(a)]  $J'(j,k)- J(j,k) \leq \varepsilon$, 
\item[(b)]  $J'(j,k)\geq a_pA_p -\varepsilon(1+C)$ for $j>k>k_2$.
\end{enumerate}

\begin{proof}[Proof of Claim]
{\bf (a)} By the quasi-continuity, we can choose an open set $U$ such that functions
 $u , v$ are continuous off the set $U$  and $\c_m(U,\Omega)<\varepsilon/2^{m+1}$.
 Then
 \begin{equation}
 \label{su-cl-1}
 	\int_U (dd^c(v_j+v_k))^m\wedge\beta^{n-m}
		<2^m cap_m(U,\Omega)
		<\varepsilon/2,
\end{equation}
\begin{equation}
\label{su-cl-2}
	\int_U (dd^c(u_j+v_k))^m\wedge\beta^{n-m}<\varepsilon/2.
\end{equation}
Therefore
\begin{equation}
\label{su-cl-3}
\begin{aligned}
J'(j,k)-J(j,k)
	& = \int_\Omega (u-u_j)dd^cv_j\wedge S(j,k)
		-\int_\Omega (u-u_j)dd^cv_k\wedge S(j,k) \\
	& = \int_\Omega v_jdd^c(u-u_j)\wedge S(j,k)
		- \int_\Omega v_kdd^c(u-u_j)\wedge S(j,k) \\
	& = \int_\Omega (v_j-v_k)dd^c(u-u_j)\wedge S(j,k)\\
	& = \int_{\Omega\setminus U} (v_j-v_k)dd^c(u-u_j)\wedge S(j,k)
		 + \int_U (v_j-v_k)dd^c(u-u_j)\wedge S(j,k)\\
	& \leq \int_{\Omega\setminus U} \|v_j-v_k\|dd^c(u+u_j)\wedge S(j,k)
		+ \int_Udd^c(u+u_j)\wedge S(j,k),
\end{aligned}
\end{equation}
where in the second equality we used the integration by parts formula twice with
$u=u_j=\varphi$, $v_j=0$ on the  boundary, and  in the  last  estimate we used the fact
$-1<u_j,v_j<0$. Since $v_j$ converges uniformly to $v$ on
$\Omega\setminus U$ one can find $l>k_1$ such that $\|v_j-v_k\|<\varepsilon/ 2C$ on
$\Omega\setminus U$ for $j>k>l>k_1$. This combined with \eqref{su-le3-3}, \eqref{su-cl-1}
and \eqref{su-cl-2}imply that each  of the integrals in the last line of \eqref{su-cl-3}
is at most $\varepsilon/2$.The first part of the claim follows.

{\bf (b)} We  first observe that from the upper bound of all $u_j$ (resp. $v_j$)
by $\sup\varphi$  (resp. $0$) on the boundary,  we have  for $k>k_2>l$,
in a neighborhood of $\partial \Omega$
 \begin{equation}
 \label{su-cl-4}
 v_k\leq v+\varepsilon \;\; \text{and} \;\; u_k\leq u+\varepsilon.
\end{equation}
Those inequalities are still valid (after increasing $k_2$) on
$\Omega\setminus U$  thanks to the Hartogs lemma.
Hence, using \eqref{su-le3-3}, \eqref{su-cl-1} and \eqref{su-cl-4} we have for $j>k>k_2$
\begin{align*} J'(j,k)
& = \int_\Omega (u-u_j)dd^cv_j\wedge S(j,k)\\
& \geq  a_p\int_{\{ u-u_j\geq a_p\}}dd^cv_j\wedge S(j,k)
	+ \int_{\{ u-u_j <a_p\}}(u-u_j)dd^cv_j\wedge S(j,k) \\
& =  a_p\int_{\{ u-u_j\geq a_p\}}dd^cv_j\wedge S(j,k)
	+ \int_{\{ u-u_j <a_p\}\cap(\Omega\setminus U)}(u-u_j)dd^cv_j\wedge S(j,k)\\
& \;\; + \int_{\{ u-u_j <a_p\}\cap U}(u-u_j)dd^cv_j\wedge S(j,k)\\
& \geq a_p\int_{\{ u-u_j\geq a_p\}}dd^cv_j\wedge S(j,k)
	- \varepsilon \int_{\Omega\setminus U}dd^cv_j\wedge S(j,k)
		-\int_Udd^cv_j\wedge S(j,k)\\
&\geq a_p A_p-\varepsilon(1+C).
\end{align*}

Thus the proof of the claim is finished.
\end{proof}

From {\it Claim} and \eqref{su-le3-4} we get
$$
	\int_{\{u-u_j\geq d\}}dd^cv_k\wedge S(j,k)
		\geq J(j,k)- dC\geq J'(j,k)-\varepsilon-dC\geq a_pA_p
			-\varepsilon(1+C)-\varepsilon-dC.
$$
If we take
\begin{equation}
\label{su-le3-5}
 a_{m+1}:=d=\frac{a_pA_p}{4C}\;\; \text{ and }\;\; \varepsilon\leq \frac{a_pA_p}{2(2+C)},
\end{equation}
then
$$
	\int_{\{u-u_j\geq d\}}dd^cv_k\wedge S
		\geq \frac{a_pA_p}{4}:=A_{p+1} \;\; \text{for} \;\; j>k>k_2,
$$
which finishes the proof of the inductive step and that of Lemma~\ref{su-le-3}.
\end{proof}

{\it End of the  proof of Theorem 2.1.} It is enough to prove the condition (2.4) in Lemma 2.2.
We argue by contradiction.  Suppose that it is not true.
Then the assumptions of Lemma 2.3 are valid and  its statement for
$p=m$ tells that for a fixed $k>k_1$
$$
	 \int_{\{u-u_j
	 	\geq a_m\}} (dd^cv_k)^m\wedge\beta^{n-m}
		>A_m \;\; \text{ when } \;\;  j>k.
$$
Thus
\begin{equation}
\label{su-pf-2}
 V(\{u-u_j\geq a_m\})
 	\geq \frac{1}{M_k}\int_{\{u-u_j\geq a_m\}}(dd^cv_k)^m\wedge\beta^{n-m}
	>\frac{A_m}{M_k}\;\; \text{for} \;\; j>k,
\end{equation}
because $(dd^cv_k)^m\wedge\beta^{n-m}=g_k\beta^n\leq M_k\beta^n$ for some $M_k>0$.
But \eqref{su-pf-2} contradicts the fact $u_j\rightarrow u$ in $ L^1_{loc}$,
i.e every subsequence of $\{u_j\}$ also converges to $u$ in $ L^1_{loc}$.
Thus, the theorem is proved under two extra assumptions.

{\it General case (we  remove two extra assumptions).} {\bf 1)}
Suppose that $\varphi\in C(\partial\Omega)$ and
the measure $\mu$ has compact support in $\Omega$.
We choose a decreasing sequence
$\varphi_k\in C^2(\partial\Omega)$ converging to $\varphi$.
Then we obtain a sequence of solutions $u_k$ satisfying
$$
\begin{cases}
	u_k\in  SH_m\cap  L^\infty(\Omega) , \\
	(dd^cu_k)^m\wedge\beta^{n-m}=\mu ,\\
	 \lim_{\zeta\rightarrow z} u_k(\zeta)
	 	=\varphi_k (z)  \text{  for any  }   z\in\partial\Omega.
\end{cases}
$$
It follows from the comparison principle, Corollary~\ref{pr-co-14}-(a),
that $u_k$ is decreasing and $u_k\geq v_0$ with $v_0$ a subsolution
without modifications. Set $u=\lim u_k$.
Then $u\geq v_0$ and $(dd^cu)^m\wedge\beta^{n-m}=\mu$ by Proposition~\ref{pr-pr-10}.
Thus, $u$ is the  required solution.

{\bf 2)} Suppose that $\mu$ is a finite positive Borel measure, $\varphi\in C(\partial\Omega)$.
Let $\chi_j$ be a non-decreasing sequence of cut-off functions
 $\chi_j\uparrow 1$ on $\Omega$. Since $\chi_j\mu$ have compact support in
 $\Omega$, one can find solutions corresponding to $\chi_j\mu$, the solutions will be
bounded from below by the given subsolution $v_0$ (from the comparison principle)
and they will decrease to the solution by the convergence theorem.
Thus we have proved Theorem~\ref{su-th-1}.
\end{proof}


%


\begin{thebibliography}{CNS}

\bibitem[BT1]{BT1}
E. Bedford, B.A. Taylor,
The Dirichlet problem for a  complex Monge-Amp\`ere equation,
\it Invent. Math.\rm {\bf 37} (1976), 1-44.

\bibitem[BT2]{BT2}
E. Bedford, B.A. Taylor,
A new capacity for plurisubharmonic functions,
\it Acta Math. \rm {\bf 149} (1982), 1-40.

\bibitem[Bl]{Bl}
Z. B\l ocki,
Weak solutions to the complex Hessian  equation,
\it Ann. Inst. Fourier (Grenoble) \rm  {\bf 55}, 5 (2005), 1735-1756.

\bibitem[CNS]{CNS}
L. Caffarelli, L. Nirenberg, J. Spruck,
The Dirichlet problem for nonlinear second order elliptic equations, III: Functions of the eigenvalues of the Hessian,
 \it Acta Math. \rm  {\bf 155} (1985), 261-301.

\bibitem[Ce1]{Ce1}
U. Cegrell, Pluricomplex energy,
 \it Acta Math. \rm {\bf 180:2} (1998), 187-217.

\bibitem [Ce2]{Ce2}
U. Cegrell,
The general definition of the complex Monge - Amp\`ere operator,
\it Ann.\ Inst.\ Fourier (Grenoble), \rm {\bf 54} (2004), 159--179.

\bibitem[D1]{D1}
J.P. Demailly,
Potential theory in several complex variable,
\it Lecture notes. \rm\ ICPAM, Nice, 1989.

\bibitem[D2]{D2}
J.P. Demailly,
Complex analytic and differential geometry,
available at {\it http://www-fourier.ujf-grenoble.fr/~demailly/books.html},  \rm\ version September 2009.

\bibitem[Di]{D}
S. Dinew,
 An inequality for mixed Monge-Amp\`ere measures,
\it Math. Zeit. \rm {\bf 262} (2009), 1-15.

\bibitem[DK]{DK}
S. Dinew, S. Ko\l odziej,
A priori estimates for the complex Hessian equations,
 \it preprint\ \rm\ arXiv1112.3063

\bibitem[Ga]{Ga}
L. G\aa rding,
An inequality for hyperbolic   polynomials,
\it J. Math. Mech. \rm {\bf 8} (1959) 957-965.



\bibitem[Ho]{Ho}
 L. H\"ormander,
Notions of convexity, Birkhauser, Boston 1994.

\bibitem[K1]{K1}
S. Ko\l odziej,  The range of the complex Monge-Amp\`ere operator II,
\it Indiana U.  Math. J. \rm  {\bf 44},  3 (1995), 765-782.

\bibitem[K2]{K2}
S. Ko\l odziej,
The complex Monge-Amp\`ere equation and pluripotential theory,
\it Memoirs Amer. Math. Soc. \rm {\bf 178} (2005) 64p.

\bibitem[Li]{Li}
S.-Y. Li,
On the Dirichlet problems for symmetric function equations of the eigenvalues of the complex Hessian,
 \it Asian J. Math. \rm\ {\bf 8} (2004), 87-106.

\bibitem[TW]{TW}
N.S. Trudinger, X.-J. Wang,
Hessian measures II,
\it Ann. of Math. \rm {\bf 150} (1999), 579-604.

\bibitem[W]{Wa2}
X.-J. Wang,
The $k$-Hessian equation,
\it Lect. Not. Math. \rm {\bf 1977} (2009).

\end{thebibliography}
\end{document}